\documentclass[12pt]{amsart}
\usepackage{geometry}
\usepackage{graphicx}
\usepackage{amssymb}
\usepackage{epstopdf}
\usepackage{amsmath,amscd}
\usepackage{amsthm}
\usepackage{url,verbatim}
\usepackage{mathtools}
\usepackage[table]{xcolor}
\usepackage[title]{appendix}

\usepackage{etoolbox}

\RequirePackage[colorlinks,citecolor=blue,urlcolor=blue]{hyperref}
\usepackage{breakurl}
\theoremstyle{plain}
\newtheorem{theorem}{Theorem}
\newtheorem{definition}[theorem]{Definition}

\newtheorem{proposition}[theorem]{Proposition}

\newtheorem{example}[theorem]{Example}
\newtheorem{remark}[theorem]{Remark}

\newtheorem{assumption}[theorem]{Assumption}

\newtheorem{question}[theorem]{Question}

\newcounter{mycount}
\newenvironment{romlist}{\begin{list}{\rm(\roman{mycount})}%
   {\usecounter{mycount}\labelwidth=1cm\itemsep 0pt}}{\end{list}}
\newenvironment{numlist}{\begin{list}{\arabic{mycount}.}%
   {\usecounter{mycount}\labelwidth=1cm\itemsep 0pt}}{\end{list}}
\newenvironment{letlist}{\begin{list}{\rm(\alph{mycount})}%
   {\usecounter{mycount}\labelwidth=1cm\itemsep 0pt}}{\end{list}}

\numberwithin{equation}{section}
\numberwithin{theorem}{section}
\numberwithin{figure}{section}

\newcommand\RR{{\mathbb R}}
\newcommand\PP{{\mathbb P}}

\newcommand\qq{\qquad}
\newcommand\q{\quad}
\newcommand\si{\sigma}
\newcommand\be{\beta}
\newcommand\al{\alpha}
\newcommand\Si{\Sigma}
\newcommand\g{\gamma}

\newcommand\De{\Delta}
\newcommand\ol{\overline}

\renewcommand\o{{\mathrm o}}

\newcommand\ZZ{{\mathbb Z}}

\newcommand\EE{{\mathbb E}}

\newcommand\eps{\epsilon}

\newcommand\resp{respectively}

\newcommand\oo{\infty}

\newcommand\TT{{\mathbb T}}

\newcommand\de{\delta}

\renewcommand\ell{l}

\newcommand\var{\mathrm{var}}

\newcommand\SZ{S\l omczy\'nski and {\. Z}yczkowski}
\newcommand\bc{\beta_{\text{\rm c}}}
\newcommand\Nmax{N_{\text{\rm max}}}
\newcommand\Nmin{N_{\text{\rm min}}}

\title[On influence and compromise]{On influence and compromise\\
in two-tier voting systems}

\author{Geoffrey R.\ Grimmett}
\address{Statistical Laboratory, Centre for
Mathematical Sciences, Cambridge University, Wilberforce Road,
Cambridge CB3 0WB, UK} 
\address{School of Mathematics \&\ Statistics, The University of Melbourne, 
Parkville, VIC 3010, Australia}
\email{g.r.grimmett@statslab.cam.ac.uk}
\urladdr{\url{http://www.statslab.cam.ac.uk/~grg/}}

\begin{document}

\begin{abstract}
We examine two aspects of the mathematical basis for two-tier voting systems, such as that of
the Council of the European Union. These aspects concern the use of square-root weights and the choice of quota.

Square-root weights originate in the Penrose square-root system, which assumes
that votes are cast independently and uniformly at random, and is based
around the concept of equality of influence of the voters across the Union. There
are (at least) two distinct definitions of influence in current use in probability theory, namely,
\emph{absolute} and \emph{conditional influence}. These are in agreement 
when the underlying random variables are independent, but
not generally. We
review their possible implications for two-tier voting systems, 
especially in the context of the so-called collective bias model. We show
that the two square-root laws invoked by Penrose are unified through the use of 
conditional influence.

In an elaboration of the square-root system, \SZ\ have proposed an exact
value for the quota $q=q^*$ to be achieved in a successful vote of a two-tier system,
and they have presented numerical and theoretical evidence in its support.  We indicate
some numerical and mathematical issues arising in the use of a Gaussian (or normal) approximation 
in this context, and we propose that other values of $q$ may be as good if not better than $q^*$.
We discuss certain aspects of the relationship between
 theoreticians and politicians in the design
of a two-tier voting system, and we reach the conclusion that the choice of 
quota in the square-root system is an issue for politicians informed by theory.
\end{abstract}

\date{21 March 2018, revised 14 May 2019}

\maketitle

\section{Introduction and background}\label{sec:intro}

\subsection{Preamble}
Mathematics is fundamental to the design and analysis of voting systems
(see, for example, the books \cite{BY1982,Birk11, FM, LV, Puk2017}). 
Mathematical models for human behaviour frequently involve probability, and they
invariably rely upon assumptions whose validity is ripe for debate. As a general rule,
the greater are the assumptions, the more precise are the conclusions. 
A balance needs be struck between tractability and applicability: 
excessive assumptions tend to undermine practical relevance, whereas \lq\lq nothing will come of nothing"
[Shakespeare].  
Assumption-based conclusions must be exposed to a robustness analysis: to what degree are they
robust when the underlying assumptions are perturbed? 

These issues are illustrated here in a study of the so-called Jagiellonian Compromise (JagCom)
of \SZ\ \cite{SZ04,SZ06,SZ07,SZ11}. The JagCom is a proposal for a two-tier voting system such as 
that of the Council of the European Union (see Section \ref{ssec13} for further details).
It is based on (i) \emph{square-root weights} coupled with (ii) a certain formula for the \emph{quota}.
While the current work was born out of an interest in learning about the JagCom, it has developed
into (i) a broader study of the notion of power (or influence) for general probability distributions, 
combined with (ii) a critical analysis of the arguments leading to the given quota.    
The conclusions of this article illuminate the balance
between theory and practical relevance.  

The political context of this paper is as follows. The debate rolls on concerning 
the allocation of seats in the European Parliament (EP) 
between the Member States of the EU (see, for example,  \cite{Cart,PukG17}).
It has been argued by members of the Constitutional Affairs Committee (AFCO) of the EP that strategic
reform of the somewhat \emph{ad hoc} method of allocation of parliamentary seats should be considered  only in parallel 
to a review of the two-tier voting system of the EU Council. 
The JagCom is a leading theoretical contender for implementation in the Council.
It has been supported in two open letters to EU governments, \cite{openl2,openl1}, signed by significant 
numbers of prominent theoreticians, and it has been the subject of a volume of
positive publicity including \cite{KSZ07,Mach07,comp,SZ10,ZSZ}. 
The discussion in Brussels is likely to intensify in the months and years to come,
and this is a propitious moment to re-examine the JagCom in some detail.

There are two principal parts to this paper, as outlined in the following two subsections.  
The first concerns the definition of the \lq power' of an individual (as introduced by Penrose) 
for general probability distributions. This is connected to the choice of weights in a two-tier voting
system such as the JagCom.  The second is a discussion of the choice of quota in the JagCom.

\subsection{Power and influence}\label{sec:1.2}
Lionel Penrose's 1946 paper \cite{Pen} is a fundamental work in the mathematical theory of voting, 
and it has received a great deal of attention. Penrose found it convenient to assume that members
of a population choose their votes \emph{independently} at random, and are 
\emph{equally likely} to choose either of the two possible outcomes. 
These assumptions of independence and unbiasedness lead to a mathematically sophisticated
theory based around the classical study of the sums of so-called independent and identically
distributed (\lq iid'), symmetric random variables (see \cite{Pet1975},
or the less sophisticated account \cite[Sect.\ 5.10]{GS01}). 
That said, independence and unbiasedness may, in practice,  be far from the truth
in specific cases.

The square-root voting system of Penrose \cite{Pen} is prominent in
discussion of two-tier voting systems in general, 
and in specific of that of the Council of the European Union
(see, for example, \cite{Kir16,SZ11}).
The challenge confronted by Penrose is to devise a system for pooling the views of a number of Member States with varying 
population sizes. What weight $w_j$ should be assigned
to the opinion of State $j$, having a population of size $N_j$? 
The Penrose system amounts to the proposal $w_j \propto \sqrt{N_j}$.
The essence of Penrose's argument is the observation that the number $H$ of heads shown 
in $N_j$ independent, unbiased coin tosses satisfies 
\begin{equation}\label{1}
\EE\left|H-\tfrac12 N_j\right| \sim \sqrt {2N_j/\pi}, \qq \text{for large } N_j.
\end{equation}
(Here and later,  $\EE$ denotes expectation, and $\PP$ denotes probability.)
We shall refer to \eqref{1} as \emph{Penrose's second square-root law}.
The reader is referred back to \cite{Pen} for the deduction of square-root weights from \eqref{1},
although s/he may prefer to read Kirsch's least-squares argument as presented in Section \ref{sec:3.2}.
(Although \eqref{1}, and the subsequent \eqref{4}, are \emph{asymptotic} relations, sharp bounds
may be obtained by elementary methods.)

\begin{remark}\label{rem12}
Penrose defines the \lq edge' as $|N_F-N_A|$, where $N_F$ (\resp, $N_A$) is
the number of votes in favour (\resp, against) the motion. It is immediate
that $|N_F-N_A|=|H-(N_j-H)| = |2H-N_j|$, so that the mean edge equals
$2 \EE|H-\frac12 N_j|$.
\end{remark}

Penrose \cite{Pen} discussed also the concept of the  \lq power' 
(termed \lq influence' in the current work, after \cite{BOL,Ru81}) 
of an individual voter within a given election or vote.
He noted that, in a vote within a State containing $N_j$ individuals whose votes are independent
and unbiased, this 
power, denoted $\al_j$, has order 
\begin{equation}\label{4}
\al_j \sim \sqrt{2/(\pi N_j)}, \qq \text{for large } N_j,
\end{equation}
(see also Banzhaf \cite{Bz}). 
We shall refer to \eqref{4} as \emph{Penrose's first square-root law}.

Only one square-root is sometimes attributed to Penrose. In \cite{Pen}, he stated \eqref{1} and 
he proved \eqref{4},
and he did not note their inter-relationship. Some later authors
have linked \eqref{1} and \eqref{4} by proposing  a weight $w_j'$ for State $j$
such that the product $\al_j w_j'$
does not depend on population-size, that is, $w_j' \propto \sqrt{N_j}$, in agreement with \eqref{1}
(see, for example, the discussion at \cite[p.\ 48 ff.]{SZ10} and \cite[Sect.\ 1]{SZ11}).
This argument appears to assume that: (i) in a population with size $N_j$ and individual
power $\al_j$, the collective power is $\al_j N_j$, and (ii) $1/\al_j$ has, generically, the same order
as $\EE|H-\tfrac12 N_j|$. The first assumption here is open to discussion, and the second
is false for general distributions. Proposition \ref{prop1} and Remark \ref{rem33} explain the 
true relationship between \eqref{1} and \eqref{4} in the context of general 
probability distributions.

The square-root laws of this article are those of Penrose \cite{Pen}. It is not the current purpose to
discuss in detail the relationship between voting power and voting weight (see, for example,
\cite{AE,FM,Napel-ox}).  

In Section \ref{sec:inf}, we introduce the notions of the \emph{absolute} and the \emph{conditional influences}
of an individual in an election. The absolute influence is that considered by Penrose and later authors;
the conditional influence is sometimes considered more appropriate in situations where 
individuals' votes are \emph{dependent} random variables. The two quantities are equal
in the independent case, but not generally so.

\subsection{The Jagiellonian Compromise, a two-tier voting system}\label{ssec13}
In a method since dubbed the \lq Jagiellonian Compromise' (JagCom), \SZ\ \cite{SZ04,SZ06,SZ07,SZ11}
have proposed the following two-tier voting system using square-root weights 
together with a  particular value $q^*$ for the quota $q$.
Writing $N_1,N_2,\dots,N_s$ for the populations of the $s$ States of the Union, 
under the JagCom a motion is passed if and only if
\begin{equation}\label{2}
\sum_{j\in J} \sqrt{N_j} - \sum_{j\in \ol J} \sqrt{N_j} > q^*W,
\end{equation}
where $J$ is the set of States voting in favour of the motion,  $\ol J$ is the set voting against, and
\begin{equation*}\label{2a}
q^*:= \frac{\sqrt N}W,\q  W= \sum_{j=1}^s \sqrt {N_j},
\q N=\sum_{j=1}^s N_j.
\end{equation*}
The value $q=q^*$ given in \eqref{2} is supported by a 
heuristic argument based on approximation
by a Gaussian distribution. Although no rigorous justification of this approximation is
yet available (see Appendix \ref{sec:4.2} of the current work), its conclusions gain 
some support using exact numerical methods (see Section \ref{sec:43}).
Note that the $q^*$ of \eqref{2a}  is not quite the quota of \cite{SZ07}, but rather that
of \cite{Kir07}.

Salient features of two-tier voting systems are summarised in Section \ref{sec:ttv}, with 
special attention to the work of Kirsch \cite{Kir07,Kir16} and \SZ\ 
\cite{SZ04,SZ06,SZ07,SZ11}. This is followed by
a discussion in Section \ref{sec:sinf} and Appendix \ref{sec:4.2} of the 
influences of the weighted States within the Council, and of the use 
and potential misuse of the Gaussian approximation
in estimating certain related probabilities.
The closing Section \ref{sec:rems} contains some reflections on the JagCom,
and in particular the following conclusions.
\begin{numlist}
\item The JagCom is derived via a set of principles that can be stated
unambiguously and analysed rigorously, and the resulting system is
robust with respect to population changes. Nevertheless,
these principles are arguably fragile and unrealistic, and insufficiently
sensitive to political realities.

\item Despite  fragility in the assumptions about voting patterns 
made by the JagCom, 
we offer no superior alternative here. The
problem of allocating weights is more than just a mathematical puzzle,
but demands a more extensive political vision. 

\item  The justification for the proposed JagCom quota $q^*$ is empirical and numerical rather
than rigorous. These numerics provide only equivocal guidance which
is satisfied by a range of values of the quota, including the simpler value $q=0$. 
Indeed, $q=0$ is the value for which individual powers are maximized. 
Given the very modest differences in performance between such values of the quota, 
the final choice is best determined through \emph{political} input.
\end{numlist}

\subsection{Resum\'e}
Certain assumptions appear to be necessary for the above analyses, and these are examined 
in the current article. There are four areas that receive special attention, namely:
\begin{letlist}
\item  the underlying model in which individuals vote according to an unbiased coin toss, 
independently of other voters [Section \ref{sec:23}],

\item an alternative interpretation of the concept of \lq  voting power' or \lq influence'
[Section \ref{sec:22}],

\item  the assumptions  of mathematical smoothness under which the Gaussian approximation is
suitable for finite populations [Appendix \ref{sec:4.2}],

\item some implications of exact computations of voting powers in two-tier systems
[Section \ref{sec:43}].
\end{letlist}
Numerous earlier authors have of course considered some of these issues, namely (a),
(c), and (d), and we mention \cite{Kir16, KLMT, KN, LM04, ER, SZ07}.

\begin{remark}[on the literature]
There is an extensive existing literature on the matters considered in this article, and the author
has attempted to include appropriate references.
Apologies are offered to those authors whose work is not listed explicitly. 
\end{remark}

\section{Absolute and conditional influence}\label{sec:inf}

\subsection{The history and context of influence}
The concept of \lq influence' is central in the probability theory of disordered systems. 
Consider a system that comprises $m$ sub-systems. These could be, for example, individual voters in an election,
nodes in a disordered medium (as in the percolation model), or
particles in a model for the ferromagnet (such as the Ising/Potts models). In studying the behaviour
of the collective system, it is often key to understand the effect of a variation within a given sub-system. 
That is, what is the probability that a  change in 
a given sub-system has a substantial effect on the collective system?

The quantification of influence is long recognised as 
being central to the understanding of complex random systems.
For example, influence in voting systems was studied by Ben-Or and Linial \cite{BOL} in
work that played an important role in stimulating a systematic theory of influence and sharp threshold with
many applications in random systems (see \cite{KalS} for a review).
In percolation theory, the influence of a node is the probability that the 
node is pivotal for a given global event
(see \eqref{eq:pivotal} for 
the definition of pivotality). Estimates for influence are key to most of the principal results
for percolation (see \cite{G99}, for example). In these two areas above, the sub-systems are generally taken
to be \emph{independent} random variables. This is, however, not so for 
a number of important processes of statistical mechanics including the Ising and Potts models,
in which the sub-systems are dependent but usually positively correlated. For such systems,
\lq influence' requires a new definition, and this is provided in \cite{GG,GG11} in the context of the
Ising and random-cluster models (see \cite{G-RC}). 
 
The origins of influence are rather older than the above work, and go back at least to the work of
Penrose \cite{Pen} in 1946 and possibly the reliability literature surveyed by Barlow
and Proschan \cite{BP65} in 1965.    The two definitions of influence, referred to above, are presented next
in the context of a voting system (we shall use the standard terminology of
probability theory and the theory of interacting systems).  

\subsection{Definitions of absolute and conditional influences}\label{sec:22}
There is a population $P$ containing $m$ individuals, and a vote is taken between two possible
outcomes, labelled $+1$ and $-1$. Each individual votes either $+1$ or $-1$. 
We write $X(i)$ for the vote of person $i$, and we assume the $X(i)$ are random variables.
The \emph{vote-vector} $X=(X(1),X(2),\dots,X(m))$ takes values in the so-called
\lq configuration space'  $\Si=\{-1,1\}^m$. 
There exists a predetermined subset $A \subseteq \Si$, and the vote is declared to \emph{pass} if 
and only if $X \in A$. It is normal to consider sets $A$ which are \emph{increasing} in that
\begin{equation}\label{eq:inc}
\si\in A,\ \si\le\si'\q \Rightarrow \q \si'\in A.
\end{equation}
The inequality $\si\le\si'$ refers to the natural partial order on $\Si$ given by
$$
\si\le\si' \q\text{if and only if}\q \si(i)\le \si'(i) \text{ for all }i \in P.
$$ 
For concreteness, \emph{we assume henceforth that $A$ is an increasing subset of $\Si$}
(that is, an \lq increasing event').

For $i\in P$ and a configuration $\si=(\si(1),\si(2),\dots,\si(m))\in\Si$, 
we define the vectors $\si^i$, $\si_i$ by
\begin{equation}\label{eq:siud}
\si^i(j) = \begin{cases} 1 &\text{if } j=i,\\ \si(j) &\text{otherwise,}\end{cases}\qq
\si_i(j) = \begin{cases} -1 &\text{if } j=i,\\ \si(j) &\text{otherwise.}\end{cases}\qq
\end{equation}
That is, $\si^i$ (\resp, $\si_i$) agrees with $\si$ except possibly at $i$,
with $i$'s vote set to $1$ (\resp, $-1$).
Individual $i$ is called \emph{pivotal} if the outcome of the vote changes when s/he changes opinion
(the words \emph{decisive} and \emph{critical} are
sometimes used in the voting literature). More formally,
$i$ is called \emph{pivotal} for the configuration  $\si$ if 
\begin{equation}\label{eq:pivotal}
\si_i\notin A, \qq \si^i \in A.
\end{equation}
This holds since $A$ is assumed increasing.

In all situations considered in this paper, the individual votes $X(i)$ are assumed 
to be identically distributed and symmetric in that
\begin{equation}\label{eq:Xsym}
\PP(X(i)=1) = \PP(X(i)=-1)=\tfrac12,
\end{equation}
where $\PP$ denotes the probability measure that governs the vote-vector $X$.
Assumptions of independence will be introduced where appropriate.

\begin{definition}\label{defn2-1}
We say that the vector $X$ is \emph{symmetric} if 
\begin{romlist}
\item $X$ and $-X$ have the same distributions, and
\item for all $i \ne j$ there exists a permutation  $\pi$ of $\{1,2,\dots,m\}$
with $\pi_i=j$ such that $X$ and $\pi X$ have the same distribution, 
where $\pi X$ denotes the 
permuted vector $(X_{\pi_1}, X_{\pi_2}, \dots, X_{\pi_m})$.
\end{romlist}
\end{definition}

Condition (ii) above is weaker than requiring 
that $X$ be exchangeable, and an example is included in Appendix \ref{ex:1}.

\begin{definition}
Let $A$ be an increasing event.
\begin{letlist}
\item The \emph{absolute influence} of voter $i$ is 
\begin{align*}
\al(i) &:= \PP(X^i\in A) - \PP(X_i \in A)\\
&= \PP(X_i\notin A,\, X^i \in A) .
\end{align*}
\item The \emph{conditional influence} of voter $i$ is 
$$
\kappa(i):= \PP(A \mid X(i)=1) - \PP(A\mid X(i)=-1).
$$
\end{letlist}
\end{definition}

The so-called \lq power index', or the Banzhaf (or Banzhaf--Penrose) power index in the impartial culture
context, is the term used by many authors for the absolute influence given above.

When $\PP$ is a product measure (that is, the $X(i)$ are independent), it may be seen that
$\al(i)=\kappa(i)$, and the common value $\al$ is termed simply
\emph{influence} by Russo \cite{Ru81} and Ben-Or and Linial \cite{BOL}
(and \emph{power} by Penrose \cite{Pen}).
Equality does not generally hold when $\PP$ is not a product measure.
The above concept of \lq conditional influence' was identified
in \cite{GG}, where it was shown to be the correct adaptation of absolute influence
in proofs of sharp-threshold theorems for certain families
of \emph{dependent} measures arising in stochastic geometry and statistical 
physics. It has featured recently in work \cite{KKLN} on so-called prediction values within probabilistic games.

\begin{remark}[Success probability]\label{rem:Puk}
The \emph{success probability} $\eta(i)$ 
of voter $i$ is the probability that $i$ votes on
the winning side, that is,
$$
\eta(i) := \PP\bigl(A \cap \{X(i)=1\}\bigr) + \PP\bigl(\ol A \cap \{X(i)=-1\}\bigr),
$$
where $\ol A$ is the complement of $A$. See, for example,  \cite{PukB14,LV}.
If the $X(i)$ satisfy \eqref{eq:Xsym},
the conditional influence is related to the  success probability
by the relation 
\begin{equation}\label{succ}
\eta(i)=\tfrac12(1+\kappa(i)),
\end{equation}
see \cite[Thm 3.2.16]{FM} and \cite[Sect.\ 2(a)]{Pen}. This relation is, in fact,
the key step in the proof of the forthcoming Proposition \ref{prop1}.
The success probabilities feature in the recent work of Kirsch \cite{Kir17}.
\end{remark}

\subsection{Examples of influences}\label{sec:23}
There follow three examples of calculations that illustrate the differences between
absolute and conditional influence. For convenience,
we assume $m=2r+1$ is an odd number, and take $A$ to be the \emph{majority event}, that is, 
\begin{equation}
\label{majev}
A=\left\{\si: \sum_i \si(i)> 0\right\}.
\end{equation}
It is clear that $A$ is an increasing set.
We shall take the $X(i)$ to be Bernoulli random variables \emph{with a shared parameter $u$
which may itself be random}. Thus, the $X(i)$ are not generally independent.

Let $U$ be a random variable taking values in the interval $(0,1)$. Conditional on the event $U=u$,
the $X(i)$ are defined to be independent random variables with
\begin{equation}\label{eq:collb}
X(i) = \begin{cases} 1 &\text{with probability }  u,\\
-1 &\text{with probability } 1-u.
\end{cases}
\end{equation}
If $U$ has a symmetric distribution (in that $U$ and $1-U$ are equally distributed),
then the ensuing vote-vector $X$ is symmetric (and, indeed, exchangeable),
and this is called the \lq collective bias' model by Kirsch \cite{Kir07,Kir16}
(see also \cite{KirL}). 
Here are three examples of collective bias in which the absolute and conditional influences vary greatly.

\begin{numlist}
\item \emph{Independent voting.}
Let $\PP(U=\frac12)=1$. The $X(i)$ are independent, unbiased Bernoulli variables,
and 
\begin{equation}\label{eq:absinf}
\al(i)=\kappa(i)=\binom {2r}r \left(\frac12\right)^{2r} 
\sim \frac1{\sqrt{\pi r}} =  \sqrt{\frac2{\pi (m-1)}} \qq \text{as } m \to\oo.
\end{equation}
\item 
\emph{Uniform bias.}
Let $U$ be uniform on the interval $(0,1)$.
Then
\begin{equation}\label{eq:unifinf}
\al(i) = \int_0^1  \binom {2r}r u^r(1-u)^r\,du =\frac1m,\qq
\kappa(i)=\frac12+\o(1).
\end{equation}
\item
\emph{Polarised bias.}
Let $\PP(U=\frac13)=\PP(U=\frac23)=\frac12$. There
exists $\gamma>0$ such that
\begin{equation}\label{eq:polbias}
\al(i)=\o(e^{-\gamma m}), \qq \kappa(i)=
\frac13 +\o(1).
\end{equation}
\end{numlist}
We remind the reader that $f(m)=\o(g(m))$ means $f(m)/g(m) \to 0$ as $m\to\oo$.
The success probabilities in Cases 1--3 follow by \eqref{succ} from the above calculations.

Cases 2 and 3 are exemplars of more general situations in which: ($2'$)  the distribution 
of $U$ on some neighbourhood of  $\frac12$ is absolutely continuous
(see, for example, \cite[p.\ 674]{GKB}), and $(3')$
$U$ is almost surely bounded away from $\frac12$ (see, for example, 
the formulation of \cite[p.\ 592]{KLMT}).

\begin{remark}\label{rem24}
Only in the case of independence does the absolute 
influence have the order of the square root $1/\sqrt m$. 
In the two other situations above, the absolute influence 
is as small as $1/m$ and $e^{-\g m}$, \resp.
This illustrates the fragility of the square-root laws \eqref{1}, \eqref{4} and 
their consequences for voting (see \cite{KMN}).
\end{remark}

Further discussion of the relationship between absolute and conditional
influence  may be found in \cite[Sect.\ 2]{GG}. A review of influence for product
measures is found at \cite{KalS}, see
also \cite[Sect.\ 4.5]{G-pgs}. Uniform bias 
was discussed in \cite{ER}, and polarised bias in \cite{GG}.
Conditional influence is essentially the prediction value of \cite{KKLN}.

\subsection{Penrose's two square-root laws unified.}
We present next an elementary application of conditional influence (the proof is found
at the end of the section). We will see its relevance in
the discussion of the Penrose square-root laws in Remark \ref{rem33}.

\begin{proposition}\label{prop1}
Let $m=2r+1$ be odd.
Assume that $X$ and $-X$ have the same distributions, and let $A$ be the majority event of \eqref{majev}. 
Then $S=\sum_{i=1}^m X(i)$ satisfies
$$
\EE|S| =  \sum_{i=1}^m \kappa(i).
$$
If $X$ is symmetric, then $\kappa=\kappa(i)$ is constant and $\EE|S|=m\kappa$.
\end{proposition}

We next interpret Proposition \ref{prop1} in the language of Penrose, see Section \ref{sec:1.2}.
Consider a population of size $N_j$, and suppose the corresponding vote-vector $X=(X(1), X(2), \dots, X(N_j))$ is
symmetric. Then $\kappa=\kappa(i)$ does not depend on $i$. The number $H$ of people voting $1$
satisfies $H=\frac12(S+N_j)$, so that, by Proposition \ref{prop1} with $m=N_j$,
\begin{equation}\label{12}
\EE\left|H-\tfrac12 N_j\right| =\tfrac12 N_j\kappa .
\end{equation}
In conclusion,  $\EE|H-\tfrac12 N_j|$ grows in the manner of $\sqrt {N_j}$ if and 
only if $\kappa$ behaves in the manner of $1/\sqrt{N_j}$.

\begin{remark}\label{rem33}
In the language of Penrose \cite{Pen},
the mean \lq edge' differs from the \emph{conditional influence} by the constant multiple $N_j$.
Thus, in the context of general distributions, 
conditional influence takes precedence over absolute influence. Penrose's argument implies that, in a two-tier
voting system, the appropriate weight of state $j$ satisfies $w_j \propto N_j \kappa$,
where $\kappa=\kappa_j$ is the conditional influence. 
In this sense, Penrose's \lq\lq two" square-root laws are in reality only one,
so long as one uses \emph{conditional} rather than \emph{absolute} influences:
the asymptotic edge addressed by the
second square-root law is simply a  multiple of the asymptotic power addressed by
the first square-root law.
When voting is truly independent, the distinction between absolute and conditional influence  is nominal only.
Seen in the light of Remark \ref{rem:Puk},
Proposition \ref{prop1} supports the thesis that, for general probability measures,
the success probability is a more central quantity than the absolute influence
(cf.\ \cite[Sect.\ 3.6]{LV}).
\end{remark}

The question arises of deciding the \lq correct' definition of influence in the general voting context.
There does not seem to be a simple answer to this somewhat philosophical question,
which lies beyond the scope of this mathematical paper. Some minor reflections are offered,
within the context of the symmetric voting model of Definition \ref{defn2-1}.   

\begin{letlist}
\item If we are trying to capture the probability that an individual can, as a
theoretical exercise in free will, affect the outcome of a vote, then
we might favour absolute influence.  This interpretation
requires stepping outside the mathematical model of Definition \ref{defn2-1},
by postulating the existence of a unique individual
P who votes independently of the rest of the population. Then the absolute influence of P equals the
probability that P's vote is pivotal. 

\item When different votes are correlated, a sample of one vote contains information about the other 
votes. If we wish to gain such information, then 
conditional influence is a way to do so. On the other hand, conditional influence contains very little information
about the power of any given individual. When votes are correlated, individual power tends
to be rather small, and sometimes so small that it ceases to have great value as a discriminator.

\item Only with the help of a probabilistic model can we calculate influence.
An analysis of the above question depends on the interpretation of \lq chance' or \lq randomness' in such a model. 
Does it make practical sense to model votes as unbiased random variables?
Possibly in response to this question, some authors have argued that the views of
voters may not themselves be considered random, but it is rather the 
\emph{proposals} that are random (see, for example, \cite[p.\ 38]{FM} and \cite[p.\ 360]{Kir07}). 
This interesting suggestion poses
some philosophical challenges in its own right, not least 
arising from correlation between the responses of a given  
voter to different proposals.
\end{letlist}

\begin{proof}[Proof of Proposition \ref{prop1}]
Let $1_A$ denote the indicator function of 
an event $A$. Then, since $X$ and $-X$ have the same distribution,
\begin{align*}
\EE|S| &= \EE(S1_{S>0}) - \EE(S1_{S<0})\\
&= 2\EE(S1_{S>0})\\
&= 2\sum_{i=1}^m \EE(X_i1_{S>0}) \\
&=  \sum_{i=1}^m\bigl[\PP(S>0\mid X_i=1) - \PP(S>0 \mid X_i=-1)\bigr]\\
&=\sum_{i=1}^m \kappa(i).
\end{align*}
Subject to symmetry, the constantness of $\kappa(i)$ holds by choosing suitable permutations
of $\{1,2,\dots.,m\}$.
\end{proof}

\section{Two-tier voting}\label{sec:ttv}

\subsection{Two-tier voting systems}
We assume there
exist $s$ States with respective populations $N_1, N_2, \dots, N_s$ 
(which we take for simplicity to be odd numbers).
States are each allowed one representative on the Council of States. 
Each State is assumed to conduct
a ballot on a given issue, and the vote of voter $i$ in State $j$ is denoted $X_j(i)\in\{-1,1\}$. The outcome of
the vote in state $j$ is taken to be
\begin{equation}\label{eq:sjdef}
\chi_j:= \begin{cases} 1 &\text{if } S_j := \sum_{i=1}^{N_j} X_j(i) >0,\\
-1 &\text{otherwise}.
\end{cases}
\end{equation}
That is, $\frac12(1+\chi_j)$ is the indicator function of the event that $S_j>0$.

\begin{assumption}[\cite{Kir07}]\label{ass1}
We assume the vectors $X_j=(X_j(i): i=1,2,\dots,N_j)$, $j=1,2,\dots,s$, are independent,  which is to say that
the votes of different States are independent. We make no assumption 
for the moment about the voters of any given State
beyond that, for given $j$, the vectors $X_j$ are symmetric in that
 $X_j$ and $-X_j$ have the same distribution.
\end{assumption}

To the State $j$ is assigned a \emph{weight} $w_j>0$, and we write $W=\sum_j w_j$
for the aggregate weight of the States. 
The representative  of state $j$ votes $\chi_j$, and the weighted sum
$$
V:=\sum_{j=1}^s w_j \chi_j, 
$$
is calculated. The motion is said to \emph{pass} if
\begin{equation}\label{eq:pass}
V > q W,
\end{equation}
and to \emph{fail} otherwise, 
where $q$ is a predetermined \emph{quota} (this is not  the quota of \cite{SZ07}, but rather that
of \cite{Kir07}, see also \eqref{2}). 
This voting system depends on the weights $w=(w_j)$ and the quota $q$,
and we refer to it as the $(w,q)$ system.

Since votes are assumed independent \emph{between} States, the
absolute and conditional influences coincide at the level of the Council.

\begin{question}\label{qn1}
How should the weights $w_j$ and the quota $q$ be chosen?
\end{question}

We summarise two approaches.

\subsection{Penrose/Kirsch and least squares \cite{Kir07,Pen}}\label{sec:3.2}

Penrose has argued that, within any given state, 
the strength of a vote is proportional to the mean \lq edge', that is, the quantity 
$\EE|N_F-N_A|$
where $N_F$ is the number voting for the successful outcome and $N_A$ 
is the number voting against.  Now, $N_F-N_A = S_j$, where $S_j$ is given in
\eqref{eq:sjdef}.
This motivates the \lq Penrose' proposal that $w_j = \EE|S_j|$.

Kirsch \cite{Kir07} has proposed choosing the $w_j$ in such a way as to 
minimise the mean sum of squared errors
$$
Q:= \EE\left(\left[\sum_{j=1}^s (S_j-w_j\chi_j)\right]^2\right).
$$
A quick proof of the following proposition is given at the end of the subsection.

\begin{proposition}[{\cite[Thm 2.1]{Kir07}}]\label{prop:kir}
Subject to Assumption \ref{ass1}, the quantity $Q$ is minimised when
$w_j = \EE|S_j|$ for $ j=1,2,\dots,s$.
\end{proposition}

Thus, Kirsch's least-squares principle leads to the Penrose solution $w_j = \EE|S_j|$,
which we call the \emph{majority rule}.
As explained by Kirsch, this motivates the choice
\begin{equation}\label{eq:kirsch}
w_j = \begin{cases} \sqrt {N_j} &\text{if there is no long-range order},\\
N_j &\text{if there is long-range order},
\end{cases}
\end{equation} 
where \lq long-range order' is interpreted
as the non-decay of correlations (see the related Appendix \ref{A-ferro}).\footnote{The term \lq long-range order'
arises in statistical mechanics in situations where random spins $\sigma_v$ are located at the vertices $v$
of a lattice, such as the $d$-dimensional integer lattice $\ZZ^d$, and the correlation between two spins
$\sigma_v$, $\sigma_w$ does not approach $0$ as the distance between $v$ and $w$ diverges.}
For example,
Case 1 of Section \ref{sec:23} has no long-range order, but Cases 2 and 3 possess long-range order. 
See also Remark \ref{rem33}.

\begin{proof}[Proof of Proposition \ref{prop:kir}]
By Assumption \ref{ass1}, 
\begin{alignat*}{3}
Q&=\var\left(\sum_j (S_j-w_j\chi_j)\right)\qq &&\text{since $\EE(S_j)=\EE(\chi_j)=0$}\\
&=\sum_{j=1}^s \var (S_j-w_j\chi_j)\qq&&\text{since the $X_j$ are independent}.
\end{alignat*}
By calculus, the last summand is a minimum when $w_j=\EE(S_j\chi_j)$ as claimed.
\end{proof}

\subsection{S\l omczy\'nski/{\. Z}yczkowski and influence \cite{SZ06, SZ07,SZ11}} \label{sec:SZ}
Let us concentrate on the situation in which the entire vote-set  
$(X_j(i) : i =1,2,\dots, N_j,\ j=1,2,\dots,s)$ is a  family of independent random variables.
By independence, the absolute and conditional influences (within States) are equal.
The influence $\al_j:=\al_j(i)$ of a member of State $j$ is (asymptotically as $N_j\to\oo$)
proportional to $1/\sqrt{N_j}$, by \eqref{eq:absinf}.
The Penrose/Kirsch proposal of Section \ref{sec:3.2} amounts to $w_j = \EE|S_j| \propto \sqrt {N_j}$.
The product $\al_j w_j$ is (asymptotically, for large $N_j$)  constant across the States.
This may be seen as evidence that the voting system with this set of weights is \lq fair' across the union of the States.

How does one calculate the so-called \lq total influence' of a given voter in the $(w,q)$ system? 
A given voter is pivotal overall if s/he is pivotal within the relevant State vote, 
and furthermore the outcome of that vote is pivotal in the Council's vote.
By Assumption \ref{ass1}, the \emph{total influence} $I_j$ of voter $i$ in State $j$ is the product
\begin{equation}\label{eq:Inf}
I_j = \al_j\be_j,
\end{equation} 
where $\be_j=\be_j(w,q)$ is the influence of State $j$ in the Council's vote.
(See \cite[p.\ 67]{FM}.)
\emph{We seek a pair $(w,q)$ such that the total influences are equal (or nearly so)
across the States $j$.}

The total influences $I_j$ of \eqref{eq:Inf} need not be proportional to
the products $\al_j w_j$ of the previous paragraph,
 since the ratios $\be_j / w_j$ are in general non-constant across the States.
A number of authors including \SZ\ \cite{SZ11} have developed the following approach.
\begin{numlist}
\item Allocate to State $j$ the weight $w_j=\sqrt{N_j}$.
\item Calculate or estimate the State-influences $\be_j$ as functions of  $(w,q)$.
\item Identify a quota $q$ such that $\be_j$ is an approximately linear function of $w_j$.
\item The ensuing products $I_j=\al_j\be_j$ are approximately constant across States.
\end{numlist}

They have proposed choosing the quota  $q$   
in \eqref{eq:pass} in such a way that, for the given weights $(w_j)$,  the sum of squared differences
$$
Q:= \sum_{j=1}^s (\ol w_j-\ol \be_j)^2
$$
is a minimum, where 
$$
\ol w_j= \frac{100 w_j}{\sum_k w_k}, \qq \ol\be_j =\frac{100 \be_j}{\sum_k \be_k},
$$
are the \lq normalised' influences and weights, \resp\ (see Table \ref{table:1}). 
They present numerical, empirical, 
and theoretical evidence that this is often achieved when $q$ is near  
\begin{equation}\label{eq:q*}
q^* := \frac {\sqrt N}{\sum_j \sqrt{N_j}}, \qq\text{where } N=\sum_{j=1}^s N_j.
\end{equation}
The theoretical foundation for this proposal lies in: (i) approximating $\be_j$ by a Gaussian integral, 
and (ii)
picking $q$ such that the integrand is close to linear in $w_j$. The latter step is achieved by finding the
point at which the $N(\mu,\si^2)$ Gaussian density function 
has an inflection, and is thus locally closest to being locally linear. This inflection is easily
found by calculus to be at $q:= \mu\pm\si$, and this leads to the formula \eqref{eq:q*}.

In summary, they argue that, when $w_j=\sqrt{N_j}$ and $q=q^*$,  
the $\be_j=\be_j(w,q)$ are close to the $w_j$, and hence 
the total influences $I_j = \al_j\be_j$ are close to the products $\al_jw_j$. Finally, 
since $\al_j \sim C/ \sqrt {N_j}$ and $w_j = \sqrt{N_j}$, the last product is asymptotically constant
across the States.

The above procedure is termed the Jagiellonian Compromise (or JagCom).
We discuss in Section \ref{sec:sinf} some aspects of the derivation of the quota $q^*$
in \eqref{eq:q*}. Some peripheral support for this choice of quota may be found  in
the analyses of \lq toy models' in \cite{SZ07, tomski}.

\begin{remark}
\label{rem34}
The weights $w_j$ are chosen first in the JagCom, and then the quota $q$ is chosen according to a
linearisation argument. It may instead
be preferable to choose the parameters $(w,q)$ in such a way
that the deviation in the total influences $I_j$ is a minimum.
See, for example, \cite{Kurz12}.
\end{remark}

\section{\lq Total influences' in a two-tier system}\label{sec:sinf}

\subsection{Total influences}
A mathematical derivation of the JagCom 
quota $q^*$, \eqref{eq:q*}, seems to require certain approximations
which we discuss next. The first issue is to identify the purpose of the analysis. Let $I_j$ be the
total influence of a member of State $j$, as in \eqref{eq:Inf}. One extreme way of
achieving the near-equality of the $I_j$ is to set the quota $q$ on the left side of
\eqref{2} to be  either $-\eps+\sum_j \sqrt{N_j}$ or its negation, where $\eps>0$ is small. 
If we insist on such unanimity, we achieve
$$
I_j =\al_j\left(\frac12\right)^s \sim \frac{C}{\sqrt{N_j}} \left(\frac12\right)^s.
$$
For large $s$, these influences are nearly equal, indeed nearly equal to $0$. Their ratios
however can be as large as $\sqrt{\Nmax/\Nmin}$, where $\Nmax$ (\resp, $\Nmin$) is 
the maximum (\resp, minimum) population size.
An alternative target is that the ratios $I_j/I_k$ be as close to unity as possible, 
and a secondary target might be that the total influences are as large as possible. 
We consider this next.

Consider a vote of the Council in which each State $k$ has a preassigned weight $w_k>0$.
Let $j \in \{1,2,\dots,s\}$. 
By \eqref{eq:pass}, State $j$ is pivotal for the outcome if:
the set $J$ of States (other than $j$) voting for the motion is such that
\begin{equation}\label{eq:3}
w_J+w_j-w_{\ol J} > qW, \qq w_J-w_j-w_{\ol J} \le qW,
\end{equation}
where $J\subseteq \{1,2,\dots,s\}\setminus\{j\}$, $\ol J = \{1,2,\dots,s\}\setminus(J\cup\{j\})$, and 
$$
w_K:= \sum_{k\in K} w_k, \qq K \subseteq\{1,2,\dots,s\}.
$$
Inequalities \eqref{eq:3} may be written in the form $qW-w_j < Z_j \le qW+w_j$ where 
\begin{equation}\label{eq:Z}
Z_j=w_J-w_{\ol J} = \sum_{k\ne j} w_k \chi_k, \qq j\in\{1,2,\dots,s\},
\end{equation}
and $(\chi_k:k=1,2,\dots,s)$ is a family of independent Bernoulli random variables with
$$
\PP(\chi_k=1)=\PP(\chi_k=-1)=\tfrac12.
$$  
Therefore, State $j$ is pivotal in the Council with probability
\begin{align}\label{eq:diff}
\be_j &:= \PP\bigl(qW-w_j < Z_j \le qW+w_j\bigr)\\
&= F_{Z_j}(qW+w_j)-F_{Z_j}(qW-w_j),\nonumber
\end{align}
where $F_{Z_j}$ is the distribution function of $Z_j$.
(Similar formulae appear in \cite[App.]{SZ07}.)

\SZ\ \cite{SZ11} argue that the $F_{Z_j}$  are 
\lq nearly' Gaussian, and they consider the appoximation
\begin{equation}\label{eq:app}
\be_j \approx 2w_j \phi_{\mu_j,\si_j}(qW)
\end{equation}
where $\phi_{\mu,\si}$ is the $N(\mu,\si^2)$ Gaussian density function, and
\begin{equation}\label{eq:sigma2}
\mu_j=\EE(Z_j) = 0,\qq \si_j^2=\var(Z_j) =\sum_{k\ne j}w_k^2.
\end{equation}
More precisely, \eqref{eq:app} expresses the hypothesis that the distributions are nearly Gaussian, and that
the approximation is sufficiently uniform to be accurate to the first order, in terms of the density function.
They argue that the approximation is most accurate when $q$ is chosen in such a way that 
$qW$ is a point of inflection of $\phi_{\mu_j,\si_j}$, and this leads to the choice
$q=q^*$, with $q^*$ as in \eqref{eq:q*}.  
It is explained at the end of Appendix \ref{sec:4.2}, to which the reader is referred for
further details, that the ensuing approximations are good
in the case of the population-sizes of the EU Member States, but no proof is known to sufficient accuracy to
permit the \emph{rigorous} deduction of the quota, $q^*$.
The currently best theoretical tool for Gaussian approximations
is the so-called Berry--Esseen bound, and this is not good enough for
our purpose here. See Table \ref{table:2}.

\subsection{The argument via numerical methods}\label{sec:43}
Once one has accepted the thesis that voters are independent and unbiased, 
there is a transparent logic to the choice of weights $w_j = \sqrt{N_j}$.
Attention then turns to the choice of quota $q$. It is shown in Appendix \ref{sec:4.2}
that the mathematical argument of \SZ\ \cite{SZ07}, while neat, is at best incomplete.
In contrast, the \emph{numerical} evidence of \cite{SZ04}, in favour of $q=q^*$, retains some persuasive power. 
Similar numerical work has been carried out for the current article using
QMV2017 population data taken from \cite{PukG17}, with the results
reported in Table \ref{table:1}. These results are exact rather than being based on simulation.

\begin{table}
\begin{center}
\definecolor{mygray}{gray}{0.9}
\newcolumntype{s}{>{\columncolor{mygray}} }
\setlength{\arrayrulewidth}{0.5mm}
 \begin{tabular}{r |s l || c c | c c sc|c c sc }
\rowcolor{mygray} \multicolumn{2}{c||}{Member State} & \multicolumn{2}{c|}{weights} & \multicolumn{3}{c|}{$q=0$} &
 \multicolumn{3}{c}{$q=q^*$}  \\
\rowcolor{mygray}  $j$ && $w_j$ & $\ol w_j$ & $\be_j$ &  $\ol\be_j$ & $\ol\be_j/\ol w_j$  & $\be_j$ & $\ol\be_j$ & $\ol\be_j/\ol w_j$\\
 \hline\
 1 & Germany &  9.059  & 9.963 &0.357&  10.414 &  1.045  &0.211& 9.937 & 0.997\\
 2 & France &  8.165  & 8.979 &0.317& 9.239  & 1.029  &0.191&  8.984 & 1.001\\
 3 & Italy  &  7.830 & 8.611 &0.302& 8.816 & 1.024 &0.183& 8.619 & 1.001\\
 4 & Spain  & 6.815 & 7.495 &0.260& 7.575 & 1.011 &0.159& 7.507 & 1.002 \\
 5 & Poland   & 6.162 & 6.777 &0.233& 6.802 & 1.004  &0.144& 6.787 & 1.001 \\
 6 & Romania & 4.445 & 4.888 &0.166& 4.839 &  0.990  &0.104& 4.891 & 1.001 \\
 7 & Netherlands  & 4.152 & 4.566 &0.155& 4.512 & 0.988   &0.097& 4.568  &  1.000\\
 8 & Belgium   & 3.360  & 3.695  &0.125& 3.636 & 0.984  &0.078& 3.696  &  1.000 \\
 9 & Greece   & 3.285 & 3.613 &0.122& 3.554 & 0.984 &0.077& 3.613  &  1.000 \\
\hline
10 & Czech Rep. & 3.232 &3.554 &0.120&  3.495 & 0.983   &0.075& 3.554   & 1.000 \\
11 & Portugal  & 3.216 & 3.537 &0.119& 3.478 & 0.983 &0.075& 3.537  &  1.000 \\
12 & Sweden  & 3.162 & 3.477 &0.117& 3.418 & 0.983  &0.074& 3.477   & 1.000 \\
13 & Hungary  & 3.135 & 3.448 &0.116& 3.389 & 0.983  &0.073& 3.447   & 1.000 \\
14 & Austria  & 2.952 & 3.246 &0.109& 3.189 & 0.982  &0.069& 3.246  &  1.000 \\
15 & Bulgaria  & 2.675 & 2.942 &0.099& 2.886 & 0.981  &0.062& 2.941   & 1.000 \\
16 & Denmark  & 2.388 & 2.626& 0.088& 2.574 &  0.980  &0.056& 2.625  &  1.000 \\
17 & Finland & 2.338  & 2.571 &0.086& 2.520 &  0.980  &0.055&  2.570  &  1.000 \\
18 & Slovakia  & 2.326 & 2.558 &0.086& 2.507 &  0.980 &0.054& 2.557  &  1.000 \\
\hline
19 & Ireland   & 2.160 & 2.375 &0.080& 2.327 &  0.980  &0.050& 2.374  &  1.000 \\
20 & Croatia  & 2.047 & 2.251 &0.076& 2.204 & 0.979   &0.048& 2.250  &  1.000 \\
21 & Lithuania  &  1.700 & 1.870 &0.063& 1.829 & 0.978  &0.040& 1.868 & 0.999 \\
22 & Slovenia   & 1.437 & 1.580 &0.053& 1.545 & 0.978  &0.034& 1.579 & 0.999 \\
23 & Latvia   & 1.403 & 1.543 &0.052& 1.508 & 0.977  &0.033& 1.542 & 0.999\\
24 & Estonia  & 1.147 & 1.261 &0.042& 1.233 & 0.978  &0.027&  1.260 & 0.999 \\
25 & Cyprus  & 0.921  & 1.013 &0.034& 0.990 & 0.977  &0.021& 1.012 & 0.999 \\
26 & Luxembourg  & 0.759 & 0.835 &0.028& 0.815 & 0.976  &0.018& 0.834 & 0.999 \\
27 & Malta  & 0.659 & 0.725 &0.024& 0.708 & 0.977 &0.015& 0.724 & 0.999 \\
\hline
& Totals & 90.930 & 100 &3.429 & 100 &&2.123& 100&\\
\end{tabular}
\end{center}
\vskip12pt
\caption{Member State $j$ has weight $w_j=\sqrt{N_j}$ and normalised weight
$\ol w_j=100w_j/W$, where $W=\sum_j w_j$.
Two values of the quota $q$ are considered, namely, $q=0$ and $q=q^*$
(see \eqref{eq:q*}). For each, the influences $\be_j$ have been computed, and the normalised influences 
$\ol\be_j=100\be_j/B$ are given above, where $B=\sum_j \be_j$. 
The ratios $\ol\be_j/\ol w_j$ are presented alongside the $\ol\be_j$.
The ratios lie in the interval $[0.976,1.045]$ when $q=0$, and in the interval 
$[0.997,1.002]$ when $q=q^*$.}
\label{table:1}
\end{table}

Table \ref{table:1} lends some support to the choice $q=q^*$. 
\begin{letlist}
\item The ratios of normalised influences $\ol\be_j$
to normalised weights $\ol w_j$ are very close to $1$ when $q=q^*$.
\item Further calculations show that the sum of squared differences $Q=\sum_j(\ol w_j-\ol\be_j)^2$, considered
as a function of $q=0, \frac12 q^*, q^*, \frac32 q^*$, is a minimum when $q=q^*$.
(More refined calculations are possible.)
\end{letlist}
We note, however, the following.
\begin{romlist}
\item  The choice $q=q^*$ lacks transparency. 
In contrast, the choice $q=0$ is simple and easy to explain.
\item The ratios  $\ol\be_j/\ol w_j$ are also close to $1$ when $q=0$. 
The agreement is not quite so perfect as when $q=q^*$,
but the differences are minor. 
\item The sum $Q$ is similarly close to $0$ when $q=0$, albeit not so close as when $q=q^*$.
\item The influences $\be_j$ are largest when $q=0$. (See also \cite[App.]{SZ07}.)
\end{romlist}

In summary, the numerics are best when $q=q^*$, but the improvements relative to
the more transparent choice of $q=0$ are minor. The numerical differences between 
these two cases (and indeed other reasonable values of $q$) are so
small that they are unlikely to be separated by any technical analysis of the type of 
Appendix \ref{sec:4.2}. (See also \cite[Fig.\ 7]{widgren}.) We conclude the following. 
\begin{numlist}
\item On the basis of the theoretical and numerical
evidence concerning the ratios $\ol\be_j/\ol w_j$, there is no convincing evidence that any one value of the quota 
is materially preferable to any other.\footnote{Large positive or negative values are 
evidently poor, but 
we consider here only values $q$ such that $qW/\sqrt N$ has order $1$.
Other choices for $q$ have been considered in, for example, \cite{Birk11,CCM}.}
\item The total influences $I_j$ are largest when $q=0$.

\end{numlist}

\section{Some remarks on the Jagiellonian Compromise}\label{sec:rems}

Theoreticians propose, politicians dispose (and certain Presidents of the United States have historically 
played on both teams). Members of each group have interests and potential conflicts.
The theoretician earns respect  through honest assessment of the virtues
(or not) of, and principles underlying, a particular proposal. They
hope that politicians will accord fair weight and balance to principled proposals irrespective of 
personal advantage.   While theoreticians are usually free of conflicts arising out of employment
within a politically aligned organization, politicians are usually heavily conflicted
(see, for example, \cite{PukG17}). 

Communication between the two groups can be challenging. 
The use of language such as \lq local limit theorem' 
and \lq Berry--Esseen bound' tends to create barriers.
Such methodology is however key to proper
study of the two-tier voting system of Sections
\ref{sec:ttv}--\ref{sec:sinf},
and practitioners have worked diligently to communicate its relevance.  

The JagCom proposes the use of square-root weights $w_j = \sqrt{N_j}$
with a specific choice of the quota $q$. 
Each of these two proposals will rightly continue to attract critical discussion.
 
The square-root weights of equation \eqref{1} and Proposition \ref{prop:kir}
may be justified if: (i) there is no bias, and (ii) there is no \lq long-range order' 
(in the language of statistical mechanics).
Each of these assumptions seems over-perfectionist. 
Issues before the Council may be systematically more popular in some States than in others,
and the consequent biases risk undermining either or both of the above two assumptions. 
The \lq collective bias' model of Kirsch and others (see Section \ref{sec:inf})
is both more flexible and more empirical, at serious cost to the square-root laws for
influence and majority (see \cite{KirL}).  Other authors have considered the effect on
weights  of introducing a concave utility function (see Koriyama et al.\ \cite{KLMT}).
Such approaches give rise to weight distributions which, in turn,
benefit from calibration against the  politics and practical workings of the Council.  
The right choice of weights is not a simple matter of finding some neat mathematics.
That said, no concrete proposal
to displace square-root weights in the JagCom is made in the current work.

Having chosen the weights, the identification of the quota is subject to similar ambiguities.
The JagCom \lq exact' 
quota $q^*$ of \eqref{eq:q*} hinges on the assumptions of square-root weights and equality of
absolute influence, in combination
with numerical data and the Gaussian approximation of Appendix \ref{sec:4.2}.
The last is unproven in the current context of the 
QMV2017 population data
of the States of the EU. 

In their favour, the proposed square-root weights and the exact 
quota $q^*$ of the JagCom have been derived via a set of principles that can be stated unambiguously and analysed
fairly rigorously, and which are robust with respect to changes in population data. 
Proponents argue that they are thus less susceptible than many
alternatives  to \lq political meddling'.

We turn now to the numerics of the JagCom quota $q^*$.
If the ratios $\ol \be_j/\ol w_j$ in Table \ref{table:1} are close to $1$, then
the total influences  $I_j  = \al_j\be_j$ of \eqref{eq:Inf} are almost constant across Member States.
As indicated in the shaded columns of the table, this holds for both $q=0$ and $q=q^*$ 
(they are nearly perfect when $q=q^*$, and very close for other values of $q$).
Indeed this holds for a range of values of $q$ including the values $0$ and $q^*$
(see also \cite{widgren}).
\emph{One may deduce that, from a practical point of view, there
is little to choose between different values of $q$. This may be a situation in
which political considerations may have the final word.}
It seems generally considered to be the case that there is no enthusiasm amongst politicians
for the choice $q=0$, on the grounds that  a body such as the EU Council should
seek a compromise between simple majority and unanimity.

Overall, the details of the JagCom rely on a number of assumptions 
that are arguably fragile and/or unrealistic.  This potential 
weakness needs to be acknowledged when making the case for the JagCom.
The JagCom is a legitimate proposal for the two-tier voting system of the Council of the EU,
whose finer details may profit from input by politicians in choosing a system 
judged to serve well the needs of the nearly 450 million residents of the 27 Member States
of the European Union without the United Kingdom.   Our closing quote (Machover \cite[Abs.]{Mach07})
accords a balanced responsibility to both theoreticians and politicians:
\lq\lq This is essentially a political matter; but a political decision
ought to be made in a theoretically enlightened way."

\appendix
\section{Circular voting}\label{ex:1}
In Definition \ref{defn2-1}, condition (ii)  is weaker than requiring 
that $X$ be exchangeable (see \cite[p.\ 324]{GS01}).
Here is a simple one-dimensional example of a random vector that is symmetric
but not exchangeable. Suppose the $m$ $(\ge 4)$
voters are distributed evenly  around a circular table. Let $Z_1,Z_2,\dots,Z_m$ 
be the outcomes of $m$ independent tosses of a fair coin that shows 
the values $\pm 1$. Let $X(i)$ be the majority value of 
$Z_{i-1}$, $Z_i$,  $Z_{i+1}$, with the convention that
$Z_{m+k}=Z_k$ for all $k$ (and a similar convention for the $X(i)$).  
It may be checked that $X(i)$ and $X(j)$  are independent if and only if 
$i$ and $j$ are distance $3$ or more away from each other. 
The joint distribution of $X$ is invariant under the rotation 
$i \mapsto i+1$ modulo $m$, and is hence
symmetric. 

Similar examples may be constructed
in two and higher dimensions. In models that incorporate
a spatial element in the relationships between individual voters, 
symmetry may be a reasonable assumption when exchangeability is not.

\section{\lq Ferromagnetic voting\rq}\label{A-ferro}
Kirsch \cite{Kir07} proposed studying the voting problem via the analogy of a ferromagnetic model, such as the
classical Ising model. He concentrated on the so-called Curie--Weiss (or mean-field) model, in which
each vertex $v$ of the complete graph has a random spin $\si_v$ taking values in $\{-1,1\}$
according to a certain probability distribution dictated by the so-called Ising model. 
The reader is referred to his paper for further details.

The analysis is especially simple in this so-called \lq mean-field' case since
the complete graph has the maximum of symmetry. 
The mean-field case is, in a sense familiar to mathematical physicists, an infinite-dimensional
model.  We note that similar results are fairly immediate for the more pertinent
finite-dimensional systems also, as follows. For concreteness, let $d \ge 1$ and let $\TT_n$ be the
$d$-dimensional torus obtained from  the square grid $\{0,1,\dots,n\}^d$ with periodic boundary conditions.
Let $\bc$ be the critical value of the inverse-temperature $\beta$ 
of the Ising model on $\ZZ^d$ (we refer to \cite{G-RC,G-pgs} for 
explanations of the model and notation). Interpreting $\si_v$ as the vote of an individual placed at 
the vertex $v$, the aggregate vote
$$
S = \sum_{v\in \TT_n} \si_v
$$
satisfies
\begin{equation}\label{eq:kirsch2}
\EE|S| \approx \begin{cases} n^{d/2} &\text{if } \be < \bc,\\
n^d &\text{if } \be>\bc.
\end{cases}
\end{equation} 
The bibliography associated with the Ising model and its ramifications is extended and complex,
and is directed mostly at the corresponding infinite-volume 
problem defined on the entire 
$d$-dimensional space $\ZZ^d$. Some of the above claims for periodic boundary conditions are
well known, and others may be derived from classical results.  
The relevant literature includes \cite{ADCS,CoxG84,Ons}.

\section{Gaussian approximation}\label{sec:4.2}
\SZ\  argue that the distribution function $F_{Z_j}$ in \eqref{eq:diff} is \lq roughly' Gaussian
with mean and variance given by \eqref{eq:sigma2}.
Motivated by the local central limit theorem 
for non-identically distributed random variables
(see \cite[p.\ 195]{GS01} and  \cite{GW}, or otherwise), 
we aspire to a Gaussian approximation of \eqref{eq:diff} of  the form
\begin{align}\label{eq:lclt}
\be_j &\approx \int_{qW-w_j}^{qW+w_j} \phi_{\si_j}(z)\, dz
\approx 2w_j\phi_{\si_j}(qW)\\
&= \frac{2w_j}{\sqrt{2\pi \si_j^2}}
\exp\left(-\frac{(qW)^2}{2\si_j^2}\right), \nonumber
\end{align}
where $\phi_\si$ is the density function of the $N(0,\si^2)$ Gaussian distribution.
(No sufficiently accurate estimate of the error in the first approximation of \eqref{eq:lclt} is currently available.)
This leads to the following approximation for the total influence of a voter in State $j$:
$$
I_j = \al_j \be_j \simeq \frac{C}{\sqrt{N_j}} \frac{2w_j}{\sqrt{\De^2-w_j^2}}
\exp\left(-\frac12\cdot\frac{(qW)^2}{(\De^2-w_j^2)}\right),
$$
where $C>0$ is an absolute constant, and 
$$
\De^2=\sum_{k=1}^s w_k^2.
$$
Let $\de=\Nmax/N$, and set $w_j=\sqrt{N_j}$.
(A similar anaysis is valid with $w_j$ set to the mean edge of Remark \ref{rem12}.)

\begin{letlist}
\item If we set $q=q^*$ as in \eqref{eq:q*}, 
we obtain the approximate inequalities
\begin{equation}\label{eq:approx}
\frac{2C}\De e^{-1/[2(1-\de)]} \preceq I_j=\al_j\be_j  \preceq \frac{2C}{\De\sqrt{1-\de}} e^{-1/2},
\qq j=1,2,\dots,s,
\end{equation}
with $\De=\sqrt N$. (The symbol $\preceq$ is used in order to indicate that
the inequalities are based on  the unquantified approximation \eqref{eq:lclt}.)
These bounds are independent of the choice of $j$, and are increasingly close to
one another in the limit as $\de\to 0$.
\item If, instead, we set $q=0$, we obtain the  inequalities \eqref{eq:approx}
with the exponential terms removed.
\end{letlist}
The exact numerical values of the $\be_j$ are calculated in Section \ref{sec:43} 
for the particular case
of the 27 Member States of the European Union post-Brexit.

The above analysis depends on two Assumptions:
\begin{numlist}
\item the normal (or Gaussian) approximation \eqref{eq:lclt} is reasonable,
\item the ratio $\de  = \Nmax/N$ is small.
\end{numlist}
Assumption 2 is unavoidable in some form, and its use within \eqref{eq:approx} is quantified therein. 
We therefore concentrate henceforth on Assumption 1. The approximation of
\eqref{eq:lclt} is a statement about a finite population, and thus one needs a rate
of convergence in the central limit theorem. The classical such result is as follows.

\begin{theorem}[Berry--Esseen \cite{Berry, Esseen,Shev}]\label{berry}
There exists $C\in[0.4906,0.5600]$ such that the following holds.
Let $X_1,X_2,\dots,X_s$ be independent random variables with
$$
\EE(X_i)=0, \q  
\EE(X_i^2)=t_i^2>0,
\q \EE (|X_i|^3)=\g_i<\oo,
$$ 
and write 
$$
\si^2=\sum_{i=1}^s t_i^2, \qq S=\frac1\si \sum_{i=1}^s X_i.
$$
Then
$$
\sup_{z\in\RR}\bigl|\PP(S \le z)-\Phi(z)\bigr| \le \frac C{\si^3} \sum_i \g_i,
$$
where $\Phi$ is the distribution function of the $N(0,1)$ distribution.
\end{theorem}

Applying this to the random variable $Z_j$ of \eqref{eq:Z}, we obtain
\begin{align}\label{eq:Fapprox}
\sup_{z\in\RR}\bigl|F_{Z_j}(z)-\Phi_{\si_j}(z)\bigr| 
&= \sup_{z\in\RR}\bigl|\PP\left(Z_j/ \si_j \le z\right)-\Phi(z)\bigr|\\
&\le C \frac{\sum_{k\ne j}w_k^3}{\Bigl(\sum_{k\ne j}w_k^2\Bigr)^{3/2}},
\nonumber\end{align}
where $\si_j^2$ is given in \eqref{eq:sigma2}, and 
$\Phi_{\si}$ is the distribution function of the $N(0,\si^2)$ distribution. 
Therefore, by \eqref{eq:diff} (see \eqref{eq:lclt}),
\begin{equation}\label{eq:betabnd}
\left|\be_j - \int_{qW-w_j}^{qW+w_j} \phi_{\si_j}(z)\, dz\right| \le
2C \frac{\sum_{k\ne j}w_k^3}{\Bigl(\sum_{k\ne j}w_k^2\Bigr)^{3/2}},
\end{equation}
where  $2C\le 1.12$.

\begin{table}
\begin{center}
\definecolor{mygray}{gray}{0.9}
\newcolumntype{s}{>{\columncolor{mygray}} }
\setlength{\arrayrulewidth}{0.5mm}

\begin{tabular}{r |s l || c c sc|c c sc }
\rowcolor{mygray} \multicolumn{2}{c||}{Member State} & \multicolumn{3}{c|}{$q=0$} &
 \multicolumn{3}{c}{$q=q^*$}\\
\rowcolor{mygray}  $j$ && exact &  JagCom & BE interval  & exact & JagCom & BE interval\\
 \hline
 1 & Germany & 0.357&  0.379 &  $[0.03,0.70]$  &0.211& 0.205 &$ [-0.13,0.54]$\\
 3 & Italy  &  0.302& 0.319 & $[-0.03, 0.65]$ &0.183& 0.178 & $[-0.17,0.52]$ \\
 5 & Poland   & 0.233& 0.244 & $[-0.11,0.59]$  &0.144& 0.141 & $[-0.21,0.49]$ \\
 7 & Netherlands  &0.155& 0.160 & $[-0.19,0.50]$   &0.097& 0.095  &  $[-0.25,0.44]$\\
 9 & Greece   &0.122& 0.126 & $[-0.22,0.47]$ &0.077& 0.075  &  $[-0.27,0.42]$ \\
\hline
11 & Portugal  & 0.119& 0.123 & $[-0.22,0.46]$ &0.075& 0.074  & $[-0.27,0.41]$ \\
13 & Hungary  & 0.116& 0.120 & $[-0.22,0.46]$  &0.073&  0.072  & $[-0.27,0.41]$ \\
15 & Bulgaria  & 0.099& 0.102 & $[-0.24,0.44]$  &0.062& 0.061  & $[-0.28,0.40]$ \\
17 & Finland & 0.086& 0.089 &  $[-0.25,0.43]$  &0.055&  0.054  &  $[-0.28,0.39]$ \\
19 & Ireland   &0.080& 0.082 &  $[-0.26,0.42]$  &0.050& 0.050  &  $[-0.29,0.39]$ \\
\hline
21 & Lithuania  &  0.063& 0.064 & $[-0.27,0.40]$  &0.040& 0.039 & $[-0.30,0.37]$ \\
23 & Latvia   & 0.052& 0.053 & $[-0.28,0.39]$  &0.033& 0.032 & $[-0.30,0.37]$\\
25 & Cyprus  & 0.034& 0.035 & $[-0.30,0.37]$  &0.021& 0.021 & $[-0.31,0.36]$ \\
27 & Malta  & 0.024& 0.025 & $[-0.31,0.36]$ &0.015& 0.015 & $[-0.32,0.35]$ \\
\hline
\end{tabular}
\end{center}
\vskip 12pt
\caption{Numerical data for odd values of $j$ and the quotas $q=0$ and $q^*$. Three values are given
in each case: the exact value of the State-influence $\be_j$ as in Table \ref{table:1}
(labelled \emph{exact}), the approximation \eqref{eq:app} (labelled \emph{JagCom}),
and the rigorous interval for the State-influence calculated by the Berry--Esseen bound \eqref{eq:betabnd}
(labelled \emph{BE interval}). Note the close agreement between the exact values and the 
JagCom approximations in all cases, and especially so when $q=q^*$. 
The rigorous Berry--Esseen bound is insufficient to rule
out even $\be_j=0$ except in the unique case $q=q^*$, $j=1$.
The calculations involving the Gaussian distribution have been performed using Microsoft Excel.}
\label{table:2}
\end{table}

\begin{example}[EU27]\label{rem5}
Suppose $s=27$ and the State populations $N_1,N_2,\dots,N_{27}$ are the QMV2017 figures for
the Member States of the EU, as in \cite[Table 1]{PukG17}. We write
$N_1> N_2>\dots>N_{27}$, so that $\Nmax=N_1$, and we choose 
$w_j = \sqrt{N_j}$ and $q=q^*$ with
$q^*$ as in \eqref{eq:q*}.

The integral on the left side of \eqref{eq:betabnd} may be expressed as
\begin{equation}\label{eq:Gint}
\int_{(\sqrt N-w_j)/\sqrt{ N-N_j}}^{(\sqrt N+w_j)/\sqrt {N-N_j}} \phi(z)\, dz,
\end{equation}
where $\phi=\phi_0$. Its numerical value decreases monotonically from $0.207$ (when $j=1$) to $0.015$ (when $j=27$). 
The Berry--Esseen bound on the right side of \eqref{eq:betabnd} takes the 
value $0.332$ (when $j=1$), $0.349$ (when
$j=5$), and $0.334$ (when $j=27$), and is monotone on each of the two intervals 
$j\in[1,5]$ and $j\in [5,27]$. The bounds are too large to yield useful information about the $\be_j$, and thus they cannot be estimated using the Berry--Esseen bound.  
In contrast, the values of the integral in \eqref{eq:Gint} are notably close to the exact values of Table \ref{table:1}.
A similar analysis is valid when $q=0$.
This numerical information is summarised in Table \ref{table:2}.

We emphasise that the above observations do not invalidate the JagCom. 
Preferable to the Berry--Esseen bound would be a sufficiently precise rate of convergence
in the local central limit theorem for discrete, non-identically distributed random variables. 
We are unaware of such a result.
\end{example}

We return to the question of the accuracy of the two approximations in \eqref{eq:lclt}
(see also \eqref{eq:app}). 
Although this has not been investigated systematically in the current work, exploratory numerical
calculations have been performed with a Union of 20 States, in which $\Nmax/N$ is approximately
that of EU27 with QMV2017. All such calculations indicate that the left side of \eqref{eq:betabnd} is 
only  rarely greater than $10^{-2}$, and this is positive evidence for the Gaussian approximation
with general population figures comparable to those of EU27. 

The approximation \eqref{eq:app} amounts to a further approximation around $q^*$,
based upon the near-linearity of the
Gaussian density function near a point of inflection. 
The empirical evidence indicates, as above, that the discrepancy is generally less than $10^{-2}$.

We conclude this appendix as follows. \emph{No mathematical proof is known
of the optimality of the choice \eqref{eq:q*} of the quota $q^*$ in the JagCom
for general population distributions.}
Even if a rate can be proved in the appropriate central limit theorem, 
it is unlikely to be sufficiently tight to justify the choice $q^*$. Numerical data is,
however, positive.

\section*{Acknowledgements} 
This work was supported in part
by the Engineering and Physical Sciences Research Council under grant EP/I03372X/1.
It was conducted partly during a visit to the Statistics Department
of the University of California at Berkeley, and was completed during a visit
to Keio University, Japan. Friedrich Pukelsheim 
has kindly made a number of valuable suggestions for this work, 
and has provided some further references.
Aernout van Enter has advised
on the bibliography of the Ising model. The author thanks Wojciech S\l omczy\'nski
for his comments on the history of the JagCom and for drawing his attention to
a number of additional references. Thanks are due to the three referees, for their
careful readings and their insightful and constructive comments. 

\providecommand{\bysame}{\leavevmode\hbox to3em{\hrulefill}\thinspace}
\providecommand{\MR}{\relax\ifhmode\unskip\space\fi MR }
\providecommand{\MRhref}[2]{%
  \href{http://www.ams.org/mathscinet-getitem?mr=#1}{#2}
}
\providecommand{\href}[2]{#2}

\end{document}